\documentclass[11pt]{amsart}
\usepackage{amsthm,latexsym,amssymb,amsmath,mathrsfs,paralist,setspace}
\usepackage[all]{xy}
\newtheorem{theorem}{Theorem}[section]
\newtheorem{lemma}[theorem]{Lemma}
\newtheorem{proposition}[theorem]{Proposition}
\newtheorem{corollary}[theorem]{Corollary}

\newcommand{\Lie}{\mathrm{Lie}}
\newcommand{\id}{\mathrm{id}}
\newcommand{\den}{\mathrm{d}}

\newcommand{\W}{\mathscr W}

\newcommand{\dd}{{\delta_L}}

\newcommand{\diff}{\Delta}

\newcommand{\Dl}{L}
\newcommand{\p}{\partial}

\newcommand{\ccc}{{d_L}}
\newcommand{\cd}{{\delta_L}}
\newcommand{\cc}{{h_L}}

\newcommand{\Exp}{{\mathrm {Exp}}}
\newcommand{\oL}{{\omega_L}}

\begin{document}

\title{The $p$-adic analytic subgroup theorem revisited}
\author[C. Fuchs and D.H. Pham]{Clemens Fuchs and Duc Hiep Pham}

\begin{abstract}
It is well-known that the W\"ustholz' analytic subgroup theorem is one of the most powerful theorems in transcendence theory. The theorem gives in a very systematic and conceptual way the transcendence of a large class of complex numbers, e.g. the transcendence of $\pi$ which is originally due to Lindemann. In this paper we revisit the $p$-adic analogue of the analytic subgroup theorem and present a proof based on the method described and developed by the authors in a recent related paper.
\end{abstract}
\maketitle


\section{Introduction}

In the 1980's, W\"ustholz formulated and proved the so-called analytic subgroup theorem which counts as one of the most striking and significant results in (complex) transcendental number theory and has many applications (see \cite {aw} or \cite{w1}). Roughly speaking, the theorem asserts that an analytic subgroup of a commutative algebraic group contains a non-trivial algebraic point if and only if it contains a non-trivial algebraic subgroup. The analytic subgroup theorem gives, as very special cases, the classical theorems of Hermite, Lindemann and Gelfond-Schneider. It also directly implies the deep results of Schneider, Baker, Masser, Coates, Lang, et altera. It is the most natural generalization in terms of the qualitative version of Baker's theorem on linear forms in logarithms of algebraic numbers.

If we replace the complex domain by the $p$-adic domain, a natural question one can now ask is ``\emph{How about a $p$-adic analogue of the analytic subgroup theorem?}". Recently, Matev formulated a $p$-adic analogue of the analytic subgroup theorem and gave a proof for the statement (see \cite{matev}) following more or less the proof of the classical analytic subgroup theorem. We owe a debt of gratitude to an anonymous referee for the observation that there is even an earlier reference: the statement has already been mentioned, without proof and in equivalent form, by Bertrand in \cite{bertrand} (see Remarque 2 on p. 37). In this paper we present our version of Matev's proof in more details based on the methods described and developed in a recent related paper by the authors (see \cite{fuchs-pham}) and give some corollaries that follow.

\section{The $p$-adic analytic subgroup theorem and some consequences}
For the reader's convenience, we recall W\"ustholz' analytic subgroup theorem in more details. Let $G$ be a commutative algebraic group defined over $\overline{\mathbb Q}$ and $\frak g$ the Lie algebra of $G$. The set of complex points $G(\mathbb C)$ of $G$ has naturally the structure of a complex Lie group whose Lie algebra $\frak g_{\mathbb C}$ is the complex vector space $\frak g\otimes_{\overline{\mathbb Q}}\mathbb C$. There is an analytic homomorphism, the so-called (complex) exponential map \[\exp_G: \frak g_{\mathbb C}\rightarrow G(\mathbb C)\] which is defined as follows. For $\xi\in \frak g_{\mathbb C}$ there exists a unique homomorphism \[\varphi_{\xi}: \mathbb C\rightarrow G(\mathbb C)\] such that its differential maps the vector field $dt$ to $\xi$. We define \[\exp_G(\xi):=\varphi_{\xi}(1).\] Let $\mathfrak b$ be a subalgebra of $\frak g$ and put $B:=\exp_G(\mathfrak b\otimes_K\mathbb C)$. Then $B$ is a Lie subgroup of $G(\mathbb C)$ and said to be an {analytic subgroup} of $G$. The following theorem is called {the analytic subgroup theorem} (see \cite [Theorem 6.1]{aw}).
\begin{theorem}[W\"ustholz] Let $B\subseteq G(\mathbb C)$ be an analytic subgroup of $G(\mathbb C)$. Then $B(\overline{\mathbb Q}):=B\cap G(\overline{\mathbb Q})$ is non-trivial if and only if there exists a non-trivial algebraic subgroup $H\subseteq G$ defined over $\overline{\mathbb Q}$ such that $H(\mathbb C)\subseteq B$.
\end{theorem}
This theorem gives a fundamental generalization of the qualitative version of Baker's theorem in terms of commutative algebraic groups. More precisely, the qualitative version of Baker's theorem is the special case when the underlying algebraic group is a torus. Furthermore, it also implies many interesting results in transcendence.
To give an example we consider the simplest application of the analytic subgroup theorem, namely Lindemann's theorem: If $\alpha$ is any non-zero complex number then not both of $\alpha$ and $e^{\alpha}$ can be algebraic. In fact, assume the contrary that both $\alpha$ and $e^{\alpha}$ are algebraic. We consider the algebraic group $G=\mathbb G_a\times \mathbb G_m$, here $\mathbb G_a$ denotes the additive group and $\mathbb G_m$ denotes the multiplicative group as usual. Then $G$ is defined over $\overline{\mathbb Q}$, and the Lie algebra of $G$ is $\mathfrak g\cong \Lie(\mathbb G_a)\times \Lie(\mathbb G_m)\cong \mathbb Q\times \mathbb Q.$ This gives $\mathfrak g\otimes_{\mathbb Q}\mathbb C\cong\mathbb C\times \mathbb C.$ The set of complex points of $G$ is $G(\mathbb C)=\mathbb C\times\mathbb C^*$ and the exponential map is given by $\exp_G: \mathbb C\times \mathbb C\rightarrow \mathbb C\times \mathbb C^*, (z,w)\mapsto (z,e^w).$ Let $\Delta=\{(z,z)\in \mathbb C\times \mathbb C\}$ be the diagonal and let $B=\exp_G(\Delta)$ be the associated analytic subgroup. Then \[(0,1)\neq (\alpha, e^{\alpha})=\exp_G(\alpha,\alpha)\in B(\overline {\mathbb Q}).\]
Hence there exists a non-trivial algebraic subgroup $H$ defined over $\overline {\mathbb Q}$ of $G$ such that $H(\mathbb C)\subseteq B$. It follows that $H(\mathbb C)=B$, i.e. $B$ is an algebraic group. We deduce that the function $e^z$ is an algebraic function, a contradiction. This shows Lindemann's theorem.

We now discuss the $p$-adic analytic subgroup theorem. Let $F$ be a complete subfield of $\mathbb C_p$ containing $\overline{\mathbb Q}$, i.e. $F=\mathbb C_p$, and let $G$ be a commutative algebraic group defined over $\overline{\mathbb Q}$. According to Section \ref{Log} there is the logarithm map \[\log_{G(F)}: G(F)_f\rightarrow \Lie(G(F))=\Lie(G)\otimes_{\overline{\mathbb Q}}F.\] We denote by $V_F$ the $F$-vector space $V\otimes_{\overline{\mathbb Q}}F$ and by $G_f(\overline{\mathbb Q})$ the set of algebraic points in $G(F)_f$, i.e. $G_f(\overline{\mathbb Q})=G(F)_f\cap G(\overline{\mathbb Q})$. The {$p$-adic analytic subgroup theorem} reads as follows:
\begin{theorem}\label{past}
Let $G$ be a commutative algebraic group of positive dimension defined over $\overline{\mathbb Q}$ and let $V\subseteq \Lie(G)$ be a non-trivial
$\overline{\mathbb Q}$-linear subspace. For $\gamma\in G_f(\overline{\mathbb Q})$ with
$0\neq\log_{G(F)}(\gamma)\in V_F$
there exists an algebraic subgroup $H\subseteq G$ of positive dimension defined over $\overline{\mathbb Q}$ such that {\rm Lie}$(H)\subseteq V$ and $\gamma\in H(\overline{\mathbb Q})$.
\end{theorem}
Replacing $G$ by its identity component $G^0$ if necessary, and noting that $\Lie(G)=\Lie (G^0)$, we may, without loss of generality, assume in the proof of Theorem \ref{past} that $G$ is connected.

As already mentioned above, the theorem has been obtained by Matev (see \cite [Theorem 1]{matev}). However, we remark that his paper is still only available as a preprint and that our proof below has been worked out independently and roughly at the same time as Matev has obtained his one. Because of this we find it worth to present also our version of the result. The proof of the theorem will be given in Section \ref{propast} by using Proposition \ref{ss} below. We get the following corollaries of the theorem whose proof will also be given in Section \ref{proof}.
\begin{corollary}\label{cor1}
Let $G$ be a commutative algebraic group of positive dimension defined over $\overline{\mathbb Q}$ and $T$ a non-empty subset of $G_f(\overline{\mathbb Q})$. Denote by $W\subseteq \Lie(G(F))$ the $\overline{\mathbb Q}$-vector space generated by all elements $\log_{G(F)}(\gamma)$ for $\gamma\in T$. Then there exists an algebraic subgroup $H$ of $G$ defined over $\overline{\mathbb Q}$ such that $\Lie(H)=W$.
\end{corollary}
\begin{corollary}\label{cor2}
Let $\varphi: X\rightarrow G$ be a homomorphism of commutative algebraic groups defined over $\overline{\mathbb Q}$ with $\dim G>0$ and $\varphi_F: X(F)\rightarrow G(F)$ the homomorphism of Lie groups over $F$ induced by $\varphi$.
If $x$ is a point in $X(F)$ such that $\varphi_F(x)$ is a non-torsion point in $G_f(\overline{\mathbb Q})$, then there exists an algebraic subgroup $H$ of $G$ of positive dimension defined over $\overline{\mathbb Q}$ satisfying:
\begin{compactitem}\item[\rm 1.] $\varphi_F(x)\in H,$ \item[\rm 2.] $\Lie(H)\subseteq \varphi_{*}(\Lie(X))$, here $\varphi_*$ is the tangent map induced by $\varphi$. \end{compactitem}
\end{corollary}
To prove the $p$-adic analytic subgroup theorem, we shall reduce it to the semistable case (for the definition of semistability see Section \ref{semi}; compare also with \cite[Theorem 5]{matev} and \cite[Lemma 6]{matev}). In this case we have the following proposition.
\begin{proposition}\label{ss}
Let $G$ be a commutative algebraic group defined over $\overline{\mathbb Q}$ of positive dimension and $V$ a proper
$\overline{\mathbb Q}$-linear subspace in $\Lie(G)$, i.e. $V$ is non-trivial and $V\neq \Lie(G)$, such that $(G,V)$ is semistable over $\overline{\mathbb Q}$. Then for any non-torsion point $\gamma$ in $G_f(\overline{\mathbb Q})$ the element $\log_{G(F)}(\gamma)$ does not belong to $V_F$.
\end{proposition}
The proposition plays a crucial role for proving Theorem \ref{past}, and the next sections up to Section \ref{pross} are devoted to prove it.

\section{Preliminaries}

\subsection{Exponential and logarithm map over $p$-adic fields}\label{Log}

Throughout this section, we denote by $F$ a complete subfield of $\mathbb C_p$. Let $G$ be a commutative algebraic group defined over $F$ and $\Lie(G)$ its Lie algebra. It is known that the set of $F$-points has a structure of Lie group over $F$ and its Lie algebra coincides with $\Lie(G)$. Furthermore, if $\varphi: G\rightarrow H$ is a homomorphism of commutative algebraic groups defined over $F$ with the tangent map $\varphi_*: \Lie(G)\rightarrow \Lie(H)$, then it induces a homomorphism of Lie groups over $F$ from $G(F)$ to $H(F)$ whose differential map is $\varphi_*$.
Recall that by \cite[Chapter III, 7.6]{Bourbaki} there is a $F$-analytic homomorphism the so-called logarithm map $\log_{G(F)}: G(F)_f\rightarrow \Lie(G)$ where $G(F)_f$ is the set of $x\in G(F)$ for which there exists a strictly increasing sequence $(n_i)$ of integers such that $x^{n_i}$ tends to the unity element of $G(F)$ as $i$ tends to $\infty$. One can show that $G(F)_f$ is an open subgroup of $G(F)$ such that the quotient group $G(F)/G(F)_f$ has no non-zero torsion elements. We also note that there is an open subgroup $U$ of $\Lie(G)$ and a $F$-analytic homomorphism $\exp_{G(F)}: U\rightarrow \exp_{G(F)}(U)\subseteq G(F)_f$ such that $\exp_{G(F)}$ is the inverse map of the map which is the restriction of $\log_{G(F)}$ to $\exp_{G(F)}(U)$. In particular, this shows that $\log_{G(F)}(\gamma)=0$ with $\gamma\in G(F)_f$ if and only if $\gamma$ is a torsion point.
We list below some important properties of logarithm maps (see \cite{z}):
\\
1. For commutative algebraic groups $G,H$ defined over $F$, the logarithm map $\log_{(G\times H)(F)}$
is the map $\log_{G(F)}\times \log_{H(F)}$ with $(G\times H)(F)_f=G(F)_f\times H(F)_f$ and with $\Lie((G\times H)(F))=\Lie(G(F))\times \Lie(H(F))$.
\\
2. For a homomorphism $\varphi: G\rightarrow H$ of commutative algebraic groups defined over $F$, the set $\varphi (G(F)_f)$ is contained in $H(F)_f$ and the diagram
\begin{displaymath}
\xymatrix{
G(F)_f\ar[r]^{\varphi}\ar[d]_{\log_{G(F)}}&
H(F)_f\ar[d]^{\log_{H(F)}}\\
\Lie(G(F))\ar[r]_{\varphi_*}&\Lie(H(F))}
\end{displaymath}
commutes.

\subsection{Semistability}\label{semi}

In this section, $K$ denotes an arbitrary field of characteristic $0$. Let $G$ be a commutative algebraic group defined over $K$ and $V$ a $K$-linear subspace of $\Lie(G)$. We recall the notion ``semistable" which is due to W\"ustholz (see \cite[Chapter 6]{aw}). We associate with $(G,V)$ the index
\begin{equation*}
\tau(G,V):=
\begin{cases} \dfrac{\dim V}{\dim G}\;\;\;\; {\textnormal {\rm if }} \dim G>0,\\
1 \hspace*{1.35cm}{\textnormal{\rm otherwise}}.
\end{cases}
\end{equation*}
The pair $(G,V)$ is called {semistable (over $K$)} if for any proper quotient $\pi: G\rightarrow H$ defined over $K$, we have $\tau(G,V)\leq \tau(H,\pi_*(V))$ where $\pi_*:\Lie(G)\rightarrow \Lie(H)$ is the tangent map induced by $\pi$.

\begin{lemma}\label{lss}
If $(G,V)$ is not semistable then there exists a proper quotient $\pi: G\rightarrow G'$ such that $(G',\pi_*(V))$ is semistable.
\end{lemma}
\begin{proof}
For $H$ an algebraic group defined over $K$ and $W$ a $K$-linear subspace of $\Lie(G)$, we say that $(H,W)\prec (G,V)$ if there is a proper quotient $\pi: G\rightarrow H$ defined over $K$ such that $\pi_*(V)=W$ and $\tau(H,W)<\tau(G,V)$. Let $\Sigma$ be the set defined by
$$\Sigma:=\{(H,W); (H,W)\prec (G,V)\}.$$
Since $(G,V)$ is not semistable the set $\Sigma$ is non-empty. We observe that the set $\{\tau(H,W); (H,W)\in \Sigma\}$ is finite, and this shows that there is an element $(G',V')$ in $\Sigma$ such that
$$\tau(G',V')=\min_{(H,W)\in\Sigma}\tau(H,W).$$
Assume now that $(G',V')$ is not semistable then there exists a proper quotient $\pi': G'\rightarrow H$ defined over $K$ such that
$$\tau(G',V')>\tau(H,\pi'_*(V')).$$
Since $(G',V')\prec (G,V)$ there is a proper quotient $\pi: G\rightarrow G'$ defined over $K$ such that $\pi_*(V)=V'$. Let $q:G\rightarrow H$ be the composition of $\pi'$ with $\pi$. Then $q$ is also a proper quotient and
$$q_*(V)=(\pi'\circ \pi)_*(V)=\pi'_*(\pi_*(V))=\pi'_*(V').$$
This shows that
$$\tau(H,q_*(V))=\tau(H,\pi'_*(V'))<\tau(G',V')<\tau(G,V),$$
in other words $(H,q_*(V))\in \Sigma$. It follows that
$$\tau(G',V')\leq \tau(H,q_*(V))=\tau (H,\pi'_*(V)).$$
This contradicts the inequality $\tau(G',V')>\tau(H,\pi'_*(V'))$ above, and the lemma is proved.
\end{proof}

\subsection{$p$-adic Schwarz lemma}

For later use we quote the following proposition, which is a $p$-adic version of the Schwarz lemma. For any complete subfield $F$ of $\mathbb C_p$ and for any real number $R\geq 0$ we set $B_F(R):=\{x\in F; |x|_p<R\}$ and $\overline B_F(R):=\{x\in F; |x|_p\leq R\}$; we will skip the subscript $F$ when $F=\mathbb C_p$. Let $f(x)=\sum_na_nx^n$ be an analytic function on $\overline B(r)$ with $r>0$. We define $|f|_r:=\sup_{n}|a_n|_pr^n=\max_{n}|a_n|_pr^n.$ Then we have the following statement.

\begin{lemma}[Proposition 3.3 in \cite{fuchs-pham}]\label{4}
Let $t\geq s$ be positive real numbers, $q$ a positive integer, $f$ an analytic function on $\overline B_F(t)$, and $\Gamma$ a finite subset of $\overline B_F(s)$ of cardinality $l\geq 2$. We define
$$\delta:=\inf\{|\gamma-\gamma'|_p; \gamma, \gamma'\in \Gamma, \gamma\neq \gamma'\}$$
and
$$\mu:=\sup\{|f^{(m)}(\gamma)|_p; m=0,\ldots,q-1, \gamma\in\Gamma\}$$
with $f^{(m)}$ the $m$-th derivative of $f$.
Put $r_p=p^{-\frac{1}{p-1}}$ and assume that $\delta\leq 1$. Then
$$|f|_s\leq \max\left\{\left(\dfrac{s}{t}\right)^{ql}|f|_{t}, \mu \left(\dfrac{s}{\delta}\right)^{ql-1}r_p^{-(q-1)}\right\}.$$
\end{lemma}

\subsection{Other preliminaries}\label{op}

Let $G$ be a commutative algebraic group defined over an algebraic number field $K$ of dimension $n>0$. It is known that $G$ can be embedded into the projective space $\mathbb P^N$ with some positive integer $N$, so that $G$ is realized biregularly as a quasi-projective subset of $\mathbb{P}^N$, and that this embedding can be chosen such that the group law $G\times G\rightarrow G$ extends to a regular action $G\times\overline{G}\rightarrow\overline{G}$ of $G$ on $\overline{G}$ with $\overline G$ a compactification of $G$ (see \cite[Section 1.3]{se} and \cite[Section I-III]{fw}, where explicit such embeddings using exponential- and Theta-functions are constructed). This allows one to associate to every element of the Lie algebra $\frak g=\Lie(G)$ an invariant vector field on $\overline{G}$ which induces a derivation from $\Gamma(U,\mathcal O_{\overline G})$ to itself; here $U$ is the open affine subset in $\overline G$ defined by $\overline{G}\cap \{X_0\neq 0\}$
with $X_0$ the first projective coordinate on $\mathbb P^N$
 (for this we refer to \cite[Section 1.3]{se} or \cite[Section 2]{w1}). Furthermore, the affine algebra $\Gamma(U,\mathcal O_{\overline G})$ is generated by $\xi_1,\ldots,\xi_N$ over $K$, where
$$\xi_i:=\left(\dfrac{X_i}{X_0}\right)\!\!\bigg\vert_U,\quad \forall i=1,\ldots,N$$
with $X_1,\ldots,X_N$ the other projective coordinates on $\mathbb P^N$. This shows that $K[\xi_1,\ldots,\xi_N]$ is stable under the action of any element in $\frak g.$
We call a map $L:\{1,\ldots,n\}\rightarrow \frak g$ a {basis} if $L(1),\ldots,L(n)$ is a basis for $\frak g$.
It follows from above that the regular vector field $L(j)$ maps $\xi_i$ to an element in $K[\xi_1,\ldots,\xi_N]$ for any $i=1,\ldots,N$ and $j=1,\ldots,n$. In other words, there are polynomials $P_{i,L(j)}$ in $N$ variables with coefficients in $K$ such that
$$L(j)\xi_i=P_{i,L(j)}(\xi_1,\ldots,\xi_N),\quad \forall i=1,\ldots,N, \forall j=1,\ldots,n.$$
This means that
$$\mathcal L_j:=L(j)(\mathcal O_K[\xi_1,\ldots,\xi_N])$$
is an $\mathcal O_K$-module in $K[\xi_1,\ldots,\xi_N]$ for any $j=1,\ldots,n$. Put $\mathcal L=\mathcal L_1+\cdots+\mathcal L_n$ and define
$$\mathcal I_L:=(\mathcal O_K[\xi_1,\ldots,\xi_N]:\mathcal L)=\{t\in \mathcal O_K; t\mathcal L\subset \mathcal O_K[\xi_1,\ldots,\xi_N]\}.$$
Then $\mathcal I_L$ is an ideal of $\mathcal O_K$ and its norm $N_{K:\mathbb Q}(\mathcal I_L)$ is an ideal in $\mathbb Z$ which has to be principal. It takes the form $(\cd)$ for some positive integer $\cd$. We call $\cd$ the {denominator} of $L$.

Denote by $\p_1,\ldots,\p_n$ the canonical basis of $\Lie(K_v^n)$ defined as $\p_ix_j=\delta_{ij}$ for all $i=1,\ldots,n$ and for all $j=1,\ldots,N$, where $\delta_{ij}$ are Kronecker's delta and $x_i$ are the coordinate functions of $K_v^n$. We define the isomorphisms of $K_v$-vector spaces
 $$\p: K_v^n\rightarrow \Lie(K_v^n),\quad x=(x_1,\ldots,x_n)\mapsto x_1\p_1+\cdots+x_n\p_n$$
and
$$\iota: \Lie(K_v^n)\rightarrow \Lie(G(K_v)),\quad \iota(\p_1)=\Dl_1,\ldots,\iota(\p_n)=\Dl_n.$$
We consider now the set $G(K_v)$ of $K_v$-points of $G$. It is known that $G(K_v)$ is a Lie group over $K_v$. By \cite[Chapter III, \S 7]{Bourbaki}, there is a map $\exp$ (which is called exponential map) defined and locally analytic on an open disk $U_v$ of $G(K_v)$. The functions
$$f_i:=\xi_i\circ \Exp,\quad i=1,\ldots,N$$
are analytic on $\Lambda_v:=(\iota\circ \p)^{-1}(U_v)$ in $K_v^n$, where $\Exp=\exp\circ \iota\circ \p$ and satisfy
$$\p_j(f_i)=L(j)(\xi_i)\circ\Exp=P_{i,L(j)}(\xi_1,\ldots,\xi_N)\circ\Exp=P_{i,L(j)}(f_1,\ldots,f_N)$$
for any $i=1,\ldots,N$ and $j=1,\ldots,n$. The map $f_L=(f_1,\ldots,f_N): \Lambda_v\rightarrow K_v^N$ is called the {normalized analytic representation} of the exponential map $\exp$ with respect to the basis $L$. We define
$$\ccc:=\max_{i,j}\deg P_{i,L(j)};\quad e_L:=v(\delta_L);\quad \cc:=\max_{i,j}h(P_{i,L(j)})$$
and
$$\oL:=\max\{1,e_L\}(\cc+\log\cd+\log\ccc),$$
here by convention, $\log \ccc =0$ if $\ccc =0$.

We quote the following two statements from \cite{fuchs-pham}.

\begin{lemma}[Lemma 3.8 in \cite{fuchs-pham}]\label{ww}
Let $L:\{1,\ldots,n\}\rightarrow \frak g$ be a basis and $P(T_1,\ldots,T_N)$ a polynomial in $N$ variables with coefficients in $K$ of total degree $\leq D$. Let $T$ be a non-negative integer and $t=(t_1,\ldots,t_n)\in\mathbb{Z}_{\geq 0}^n$ be such that $T=t_1+\cdots+t_{n}$. There exists a polynomial
$P_t\in K[T_1,\ldots,T_N]$ such that
$$\partial^tP(f_1,\ldots,f_N)=P_t(f_1,\ldots,f_N),$$
satisfying\\
\hspace*{0.8cm} {\rm 1.} $\deg P_t\leq D+T(\ccc-1),$ \\ \hspace*{0.8cm} {\rm 2.} $\log|P_t|_v\ll \log|P|_v+T(\cc+\log (D+T\ccc)), \quad\forall v\in M_K,$\\
where $\partial^m:=\partial_1^{m_1}\cdots \partial_k^{m_k}$ for $m=(m_1,\ldots,m_k)\in\mathbb Z_{\geq 0}^k$ with $0\leq k\leq n$.
\end{lemma}
\begin{lemma}[Proposition 3.11 in \cite{fuchs-pham}]\label{pro1}
The functions $f_i$ satisfy
$$|f_i(x)|_p<1,\quad \forall x\in B^n(|\dd|_pr_p).$$
\end{lemma}

\section{Proofs}\label{proof}
With notations as in Theorem \ref{past}, we put $n=\dim G$ and $k=\dim_{\overline{\mathbb Q}}V$.
\subsection{Choice of basis}\label{basis}
Since $G$ is defined over $\overline{\mathbb Q}$, $V$ is a $\overline{\mathbb Q}$-linear subspace of $\Lie(G)$ and since $\gamma\in G(\overline{\mathbb Q})$ there exists an algebraic number field $K'$ such that $G,V$ are defined over $K'$ and such that $\gamma\in G(K')$. Let $\Gamma(\gamma)=\{\gamma^i; i\in\mathbb Z\}$ be the subgroup in $G(K')$ generated by $\gamma$.

\begin{lemma}
There exists an embedding $\psi: G\rightarrow \mathbb P^N$ defined over a finite field extension of $K'$ such that $\psi(e)=(1:0:\ldots:0)$ and $X_0(\psi(g))\neq 0$ for all $g\in G(K')$; here $e$ is the unity element of $G$.
\end{lemma}
\begin{proof}
This is Lemma 4.1 in \cite{fuchs-pham}; for the convenience of the reader we repeat the proof.
Let $\varphi: G\rightarrow \mathbb P^N$ be an embedding as in Section \ref{op}. Without loss of generality, one may assume that $\varphi(e)=(1:0:\ldots:0)$. We choose a field extension $K_1$ of $K'$ of degree $N+1$ and a basis $\epsilon_0,\ldots,\epsilon_N$ of $K_1$ over $K'$. It is clear that the vectors
$$(\epsilon_0,0,\ldots,0), (-\epsilon_1,\epsilon_0,0,\ldots,0),\ldots,(-\epsilon_N,0,\ldots,0,\epsilon_0)$$
form a basis of $K_1^{N+1}$ which gives rise to a unique element in GL$_{N+1}(K_1)$ mapping this basis to the standard basis of $K_1^{N+1}$. This linear isomorphism is expressed explicitly by the matrix
\begin{equation*}
A=
\begin{pmatrix} \epsilon_0^{-1}&\epsilon_0^{-2}\epsilon_1&\ldots &\epsilon_0^{-2}\epsilon_N
\\
0&\epsilon_0^{-1}&\ldots &0
\\
\vdots &\vdots &\ddots &\vdots
\\
0 & 0 &\ldots &\epsilon_0^{-1}
\end{pmatrix}.
\end{equation*}
We let $\psi$ be the composition of $A$ with the embedding $\varphi$ above. Then $\psi(e)$ has projective coordinates $(1:0:\ldots:0)$ and $X_0(\psi(g))\neq 0$ for all $g\in G(K')$. Indeed, let $(x_0:x_1:\ldots:x_N)$ be a projective coordinate of $\varphi(g)$. By the construction of $\psi$, we obtain
\begin{equation*}
\begin{split}
\psi(g)&=(\epsilon_0^{-1}x_0+\epsilon_0^{-2}\epsilon_1x_1+\cdots+\epsilon_0^{-2}\epsilon_Nx_N:\epsilon_0^{-1}x_1:\ldots:\epsilon_0^{-1}x_N)
\\
&=(\epsilon_0x_0+\epsilon_1x_1+\cdots+\epsilon_Nx_N:\epsilon_0x_1:\ldots:\epsilon_0x_N).
\end{split}
\end{equation*}
Thus we see that $\psi(e)=(1:0:\ldots:0)$. In addition, since $\epsilon_0,\ldots,\epsilon_N$ is a basis of $K_1$ over $K'$ and $x_0,\ldots,x_N$ are in $K'$, not all zero, it follows that $X_0(\psi(g))$ is non-zero.
Note that the embedding $\psi$ is defined over $K_1$.
\end{proof}
\begin{lemma}There are bihomogeneous polynomials $E_0,\ldots,E_N$ of bidegree $(a,a)$ with coefficients in a finite field extension of $K'$ and a Zariski open subset $U\subseteq G\times G$ containing $\Gamma(\gamma)\times \Gamma(\gamma)$ such that for $(g,g')\in U$ the homogeneous coordinates of $g+g'$ are $$(E_0(g,g'):\ldots:E_N(g,g')).$$
With such $E_1,\ldots,E_N$, we call $E=(E_1,\ldots,E_N)$ an addition formula for $G$.
\end{lemma}
\begin{proof}
This follows by \cite[Section 2]{w1}.
\end{proof}

These two lemmas mean that there exists an algebraic number field $K$ which contains $K'$ and satisfies the following properties:
There is an embedding over $K$ of $G$ into $\mathbb P^N$ such that the projective coordinate of the unity element of $G$ is $(1:0:\ldots:0)$ and such that the first projective coordinate $X_0(\gamma')$ for all $\gamma'$ in $\Gamma(\gamma)$ is non-zero, and there is an addition formula $E=(E_1,\ldots,E_N)$ with coefficients in $K$ of $G$.

Our discussion has shown that we may assume the group $G$, the Lie algebra $\frak g:=\Lie(G)$, the vector space $V$ and the addition formula $E$ are defined over $K$.
The inclusion $K\subset \overline{\mathbb Q}\subset \mathbb C_p$ induces a non-archimedean valuation $v$ on $K$ dividing $p$. Denote by $K_v$ the completion of $K$ with respect to $v$, then $K_v$ is a complete subfield of $\mathbb C_p$.

We fix a basis $L$ of $\frak g$ and let $f_L=(f_1,\ldots,f_N)$ be the normalized analytic representation of the exponential map $\exp_{G(K_v)}$ with respect to $L$. We identify $\frak g$ with $K^n$ (by means of the basis $L$) and consider $V$ as a $K$-linear subspace of $K^n$. It is possible to find a basis $e_1,\ldots,e_k$ for $V$ such that
$$e_i=(x_{i1},\ldots,x_{in})\in \mathcal O_K^n,\quad \forall i=1,\ldots,k.$$
We also fix a positive integer $\alpha$ large enough such that $r:=p^{-\alpha}<|\cd|_pr_p$ with $r_p:=p^{-\frac 1{p-1}}$ and such that the ball $B^n(r):=\{x=(x_1,\ldots,x_n)\in \mathbb C_p^n; |x_i|_p<r, \forall i=1,\ldots,n\}$ is contained in the domain of the map $\Exp$. The functions $f_1,\ldots,f_N$ are analytic in $B^n(r)$ and by Lemma \ref{pro1} satisfy
$$|f_i(x)|_p<1,\quad\forall x\in B^n(r),\quad \forall i=1,\ldots,N.$$
This gives the corresponding differential operators
$$\Delta_{i}=x_{i1}\partial_1+\cdots+x_{in}\partial_n$$
in $\Lie(K^n)$ with $\partial_1,\ldots,\partial_n$ the canonical basis of $\Lie(K_v^n)$.
Put
$$b:=\max\{h(x_{ij}); 1\leq i\leq k, 1\leq j\leq n\},\quad B:=e^b,$$
where $h$ denotes the (Weil) logarithmic height.
Define $d$ as the degree of the extension of $K$ over $\mathbb Q$. From now on for each $t=(t_1,\ldots,t_k)\in\mathbb Z_{\geq 0}^k$ we write $\Delta^t=\Delta_1^{t_1}\cdots\Delta_k^{t_k}$ and for each $t=(t_1,\ldots,t_m)\in \mathbb Z^m$ we write $|t|=t_1+\cdots+t_m$.

\subsection{The auxiliary function}

Let $P$ be a polynomial in $N+1$ variables $X_0,\ldots,X_N$ in $\mathbb C_p[X_0,\ldots,X_N]$. We associate with $P$ the analytic function $\Phi_P: B^n(r)\rightarrow \mathbb C_p$ which is given by
\[\Phi_P(x):=P(1,f_1(x),\ldots,f_N(x)),\quad\forall x\in B^n(r).\]
Let $g\in G(K_v)$ be defined by $g=\Exp(x)$ with $x\in B^n(r)$. We define the order of $P$ at $g$ along $V$ to be infinity, if $\Phi_P$ is identically zero in a neighborhood of $x$, and to be the order of $\Phi_P$ at $x$ along $V_{K_v}:=V\otimes_KK_v$, otherwise (we refer the reader to Section 3.6 in \cite{fuchs-pham} for more detailed description). Since $\gamma\in G(K_v)_f$ and since $\Exp(B(r))$ is an open subgroup in $G(K_v)_f$ there is a positive integer $m$ such that $\gamma^m\in \Exp(B(r))$. In order to prove Proposition \ref{ss}, we may therefore assume that $\gamma\in \Exp(B(r))$. Let $u$ be the element of $B(r)$ such that $\Exp(u)=\gamma$ and put $h:=\max\{1,h(\gamma)\}$. Throughout this paper, constants do not depend on $b$, $h$ and $p$ and we also write $A\ll B$ (resp. $A\gg B$) if there is a positive constant $c$ such that $A\leq cB$ (resp. $A\geq cB$).  Let $S_0,D,T$ be positive integers. We use the notations $c_1,c_2,...$ for real positive constants which are independent of $b,h$ and $p$. Denote by $h(P)$ the height of the polynomial $P$. We get the following proposition. We remark that the proof will be similar to that of \cite[Proposition 4.2]{fuchs-pham}, but it will be simpler because the additive group $\mathbb G_a$ does not appear.
\begin{proposition}\label{sg1}
There are positive constants $c_1$ and $c_2$ such that if $D^n\geq c_1S_0T^{k}$ there is a homogeneous polynomial $P$ in $X_0,\ldots,X_N$ of degree $D$ with coefficients in $\mathcal O_K$ such that
\begin{compactitem}\item[\rm 1.] $P$ does not vanish identically on $G$,
\item[\rm 2.] $(\diff^{t}\Phi_P)(su)=0, \forall 0\leq s<S_0, \forall t=(t_1,\ldots,t_k)\in\mathbb Z_{\geq 0}^k, t_1,\ldots,t_k<2T$,
\item[\rm 3.] $h(P)\leq c_2\big(T(\cc+\log\cd+\log(D+T\log\ccc)+b)+DS_0^2h\big).$
\end{compactitem}
\end{proposition}
\begin{proof}
 Since the dimension of $G$ is $n$, without loss of generality, we may assume that the homogeneous coordinates $X_0,\ldots,X_{n}$ are algebraically independent modulo the ideal of $G$. We shall construct a non-zero homogeneous polynomial $P$ in $n+1$ variables $X_0,\ldots, X_n$ of degree $D$ (this polynomial therefore satisfies 1. in the proposition) such that 2. and 3. in the proposition are also satisfied. We write
$$P(X)=P(X_0,\ldots,X_n)=\sum_{i=1}^{D_1}p_{i}M_i(X_0,\ldots,X_n),$$
where $M_1,\ldots,M_{D_1}$ run through all homogeneous monomials of degree $D$ in $n+1$ variables $X_0,\ldots,X_n$.
Let $E=(E_1,\ldots,E_N)$ be the addition formula for $G$ given in Section \ref{basis}. By abuse of notation we put
$$E_i(z,x):=E_i(1,f_1(z),\ldots,f_N(z), 1, f_1(x),\ldots,f_N(x)),$$
for $z,x$ in $B^n(r)$ and define
$$\Phi_{s}(x):=\Phi_P(su+x)E_0(su,x)^D.$$
For $I:=\{(s,t); 0\leq s<S_0, t=(t_1,\ldots,t_k), 0\leq t_1,\ldots,t_k<2T\},$
we shall determine coefficients $p_{i}$ such that
$$(\diff^{t}\Phi_{s})(0)=0,\quad \forall (s,t)\in I.$$
We express $\Phi_s(x)$ as
\begin{equation*}\begin{split}
\Phi_{s}(x)=\Phi_P(su+x)E_0(su,x)^D=\sum_{i}p_{i}Q_{i,s}(f_1(x),\ldots,f_N(x))
\end{split}\end{equation*}
for polynomials $Q_{i,s}$  in $N$ variables with the property that
there are positive integers $\den_s\ll s^2h$ such that
\begin{equation}\label{pt1}
\den_s\cd^{|m|}\big(\p^{m}Q_{i,s}(f_1,\ldots,f_N)\big)(0)\in \mathcal O_K
\end{equation}
and
\begin{equation}\label{pt2}
\log\Big|\Big(\p^{m}(Q_{i,s}(f_1,\ldots,f_N)\Big)(0)\Big|_v\ll |m|(\cc+\log(D+|m|\ccc))+Ds^2h
\end{equation}
for any $m=(m_1,\ldots,m_n)\in\mathbb Z_{\geq 0}^n$ and for any $v\in M_K$, where $\p^m$ denotes the differential operator $\partial_1^{m_1}\cdots\partial_n^{m_n}$ and $M_K$ denotes the set of all places on $K$.
For $i=1,\ldots,D_1$ and $(s,t)\in I$ we expand
\begin{equation*}\begin{split}
a^{st}_{i}:&=\big(\Delta^t(Q_{i,s}(f_1,\ldots,f_N))\big)(0)\\
&=\big(\diff_{1}^{t_1}\cdots \diff_{k}^{t_k}(Q_{i,s}(f_1,\ldots,f_N))\big)(0)\\
&=\Big(\prod_{i=1}^k(x_{i1}\partial_1+\cdots+x_{in}\partial_{n})^{t_i}(Q_{i,s}(f_1,\ldots,f_N))\Big)(0)\\
&=\sum_{j_1=1}^{n}\cdots\sum_{j_{|t|}=1}^{n} x_{i_1,j_1}\cdots x_{i_{|t|},j_{|t|}}\partial_{j_1}\cdots\partial_{j_{|t|}}(Q_{i,s}(f_1,\ldots,f_N))\big)(0),\\
\end{split}\end{equation*}
with $(i_1,\ldots,i_{|t|})=(1,\ldots,1,2,\ldots,2,\ldots,k,\ldots,k)$ where $\iota$ repeats $t_\iota$ times for $\iota=1,\ldots,k$.
By $(2)$ we obtain
$$\log\Big|\Big(\partial_{j_1}\cdots\partial_{j_{|t|}}(Q_{j,s}(f_1,\ldots,f_N))\Big)(0)\Big|_v\ll T(\cc+\log(D+T\ccc))+Ds^2h$$
which, together with the inequalities
$$\sum_{j_1=1}^{n}\cdots\sum_{j_{|t|}=1}^{n}\Big| x_{i_1,j_1}\cdots x_{i_{|t|},j_{|t|}}\Big|_v\leq \prod_{i=1}^k\Big(\sum_{j=1}^n|x_{ij}|_v\Big)^{t_i}\leq (nB)^{|t|},\quad\forall v\in M_K,$$
 gives
$$\log|a^{st}_{i}|_v\ll  T(\cc+\log(D+T\ccc)+b)+Ds^2h.$$
Combining this with $(1)$, we see that $\den_s\cd^{2nT}a_{i}^{st}$ is in $\mathcal O_K$ and
$$\log|\den_s\cd^{2nT}a_{i}^{st}|_v\ll T(\log\cd+\cc+\log(D+T\log\ccc)+b)+DS_0^2h$$
for $(s,t)\in I$ and for $v\in M_K^\infty$. We now consider the linear forms
$$l_{st}:=\sum_{i=1}^{D_1}b^{st}_{i}T_{i}$$
in $n_0:=D_1=\binom{D+n}{n}$ variables,
where $b^{st}_{i}:=\den_s\cd^{2nT}a_{i}^{st}$ for $(s,t)\in I$.
The number $m_0$ of these linear forms satisfies $m_0\ll S_0T^k$.
Since $b^{st}_{i}\in \mathcal O_K$ we get
$$\sum_{v\in M_K^\infty}\log\max_{i}|b^{st}_{i}|_v\ll T(\cc+\log\cd+\log (D+T\log\ccc)+b)+DS_0^2h.$$
We now apply Siegel's lemma (cf. \cite[Corollary 11]{bv}). It follows that under the condition $D^n\gg S_0T^k$
there is a non-zero vector $p_0=(p_{i})$ with coordinates in $\mathcal O_K$ such that $l_{st}(p_0)=0$ for all $(s,t)\in I$ and
such that
$$h(P)\ll T(\cc+\log\cd+\log(D+T\ccc)+b)+DS_0^2h.$$
It remains to show that $(\diff^{t}\Phi_P)(su)=0$. In fact, since $l_{st}(p_0)=0$ one gets $(\diff^{t}\Phi_{s})(0)=0$ for $(s,t)\in I$.
Put
$$\Phi^*_{s}(x):=\Phi_P(su+x), \quad E_{s}(x):=E_0(su,x)^D$$
then $\Phi^*_{s}=\Phi_{s}E_{s}^{-D}$. We therefore deduce by Leibniz' rule for derivations that
\begin{equation*}\begin{split}(\diff^{t}\Phi_P)(su)=(\diff^{t}\Phi^*_{s})(0)=\big(\diff^{t}(\Phi_{s}E_{s}^{-D})\big)(0)=0.\end{split}\end{equation*}
This completes the proof.
\end{proof}
\begin{proposition} \label{lowerbound1}
Let $S$ be a positive integer and $s$ an integer such that $0\leq s<S$. Let $P$ be a polynomial determined as in the above Proposition. Assume that $\Phi$ has a zero at $su$ of exact order $T'$ along $\W$ for some positive integer $T'$. Let $t$ be any element in $\mathbb Z^k_{\geq 0}$ with $|t|=T'$ such that $(\diff^t\Phi_P)(su)\neq 0$, then
$$\log|(\diff^t\Phi_P)(su)|_p> -c_3(T'(\cc+\log\cd+\log(D+T'\ccc)+b)+DS^2h)$$
for some positive constant $c_3$.
\end{proposition}
\begin{proof}
We define
$$\Phi^*_s(x):=\Phi_P(su+x),\quad E_{s}(x):=E_0(su,x), \quad \Phi_s(y,x):=\Phi^*_s(y,x)E_{s}(x)^D.$$
By our assumption
$$0=(\diff^{\tau}\Phi_P)(su)=(\diff^\tau\Phi_P)(su)=(\diff^{\tau}\Phi^*_s)(0)$$
for $\tau\in\mathbb Z_{\geq 0}^n$ with $|\tau|<T'$
and Leibniz' rule gives
$$(\diff^{\tau}\Phi_s)(0)=\big(\diff^{\tau}(\Phi^*_sE_s^D)\big)(0)=0.$$
Using Leibniz' rule again, one gets
$$(\diff^t\Phi_P)(su)=\big(\diff^t(\Phi_sE_s^{-D})\big)(0)=(\diff^{t}\Phi_{s})(0)E_{s}^{-D}(0).$$
The same arguments as in the proof of Proposition \ref{sg1} (just replace $S_0$ by $S$) show that
$$h((\diff^t\Phi_{s})(0,0))\ll T'(\cc+\log\cd+\log(D+T'\log\ccc)+b)+DS^2h.$$
Furthermore
$$h(E_s^{-D}(0))=h(E_0(su,0)^{-D})=Dh(E_0(su,0))\ll DS^2h.$$
Since $(\diff^{t}\Phi_P)(su)\neq 0$, Liouville's inequality gives
\begin{equation*}\begin{split}
\log|(\diff^t\Phi_P)(su)|_p&\gg -h((\diff^t\Phi_P)(su))=-h\big((\diff^t\Phi_{s})(0)E_s(0)^{-D}\big)\\
&\gg -\big(T'(\cc+\log\cd+\log(D+T'\ccc)+b)+DS^2h\big),
\end{split}\end{equation*}
and the proposition follows.
\end{proof}

\subsection{Proof of Proposition \ref{ss}}\label{pross}

Assume on the contrary that $\log_{G(F)}(\gamma)\in V_F$.  Since $K_v$ is a subfield of $F$ it follows that \[\log_{G(F)}(\gamma)=\log_{G(K_v)}(\gamma)\in \Lie(G(K_v)).\]
This means that  $\log_{G(K_v)}(\gamma)$ is an element in
\begin{equation*}
V_F\cap \Lie(G(K_v))=\big(V\otimes_KF\big)\cap(\frak g\otimes_KK_v)=V\otimes_KK_v.
\end{equation*}
In other words the element $u$ belongs to the $K_v$-vector space generated by the basis $e_1,\ldots,e_k$. We express $u$ in terms of this basis and get \[u=(u_1,\ldots,u_n)=w_{1}e_1+\cdots+w_{k}e_k\] for $w_{1},\ldots, w_{k}\in K_v$. By Section \ref{basis}, we write $e_j=(x_{j1},\ldots,x_{jn})$ for $j=1,\ldots,n$ and calculate that \[u=\sum_{j=1}^kw_{j}(x_{j1},\ldots,x_{jn})=\Big(\sum_{j=1}^kw_{j}x_{j1},\ldots,\sum_{j=1}^kw_{j}x_{jn}\Big)\] which means that \[u_i=\sum_{j=1}^kw_{j}x_{ji},\quad \forall i=1,\ldots,n.\]
Let $c$ be a sufficiently large positive constant and let $[x]$ denote the largest integer less than or equal to $x$. Let $\omega_L$ be the quantity defined in Section \ref{op}. We choose
\begin{equation*}\begin{split}
S_0&=[c\omega_Lb\log h],\\
S&=[c^2S_0],\\
D&=\Big[c^{\frac{5d+1}{n-d}}S_0^{\frac{n+1}{n-d}}h^{\frac{d}{n-d}}\Big],\\
T&=\Big[c^{\frac{4d+n+1}{n-d}}S_0^{\frac{n+1}{n-d}}h^{\frac{n}{n-d}}\Big].
\end{split}\end{equation*}
Our parameters satisfies $D^n\geq c_1S_0T^k$. Hence, there exists a polynomial $P$ with properties as in Proposition \ref {sg1} which we express as $P=\sum_{i} p_iM_i$ with $M_i$ monomials in $N+1$ variables. We write $\Phi$ instead of $\Phi_P$. For $t=(t_1,\ldots,t_k)\in\mathbb Z^k_{\geq 0}$ such that $|t|<T$ we define the analytic function \[f(z):=(\Delta^{t}\Phi)(zu)\] in the variable $z\in B(r)$. For $0\leq \tau<T$ we apply the composition rule for derivatives and get
\begin{equation*}\begin{split}
f^{(\tau)}(z)&=\big((u_1\partial_1+\cdots+u_n\partial_n)^\tau\Delta^{t}\Phi\big)(zu).
\end{split}\end{equation*}
We use the representations \[u_i=\sum_{j=1}^kw_{j}x_{ji},\quad \forall i=1,\ldots,n\] to write
\begin{equation*} \sum_{i=1}^nu_i\partial_i=\sum_{i=1}^n\Big(\sum_{j=1}^kw_{j}x_{ji}\Big)\partial_i=\sum_{j=1}^kw_{j}\Big(\sum_{i=1}^nx_{ji}\partial_i\Big)=\sum_{j=1}^kw_{j}\Delta_j.
\end{equation*}
This leads to \[f^{(\tau)}(z)=\big((w_{1}\Delta_1+\cdots+w_{k}\Delta_k)^\tau\Delta^{t}\Phi\big)(zu).\] Applying multinomial expansion and Proposition \ref{sg1} one gets \[f^{(\tau)}(s)=0,\quad \forall 0\leq s<S_0.\]
Define $\overline B(R):=\{x\in\mathbb C_p; |x|_p\leq R\}$, we shall show that $f$ is analytic in $\overline B(R)$ for $R:=p^{\frac{1}{2d}}$. To prove this, we observe that $v(K_v)=(1/d_v)\mathbb Z$ with $d_v:=[K_v:\mathbb Q_p]$. This means that \[v(u_i)=\dfrac{a_i}{d_v},\] for some $a_i\in\mathbb Z_{\geq 0}$. Since $u\in B^n(r)$ it follows that $v(u_i)-\alpha>0$, i.e. $a_i-d_v\alpha>0$. This together with the inequality $d=[K:\mathbb Q]\geq [K_v:\mathbb Q_p]=d_v$ implies that \[v(u_i)-\alpha=\dfrac{a_i-d_v\alpha}{d_v}\geq \dfrac{1}{d_v}\geq \dfrac{1}{d}.\]
Hence \[R|u_i|_p=p^{\frac{1}{2d}}p^{-v(u_i)}=p^{-\frac{1}{2d}}p^{-(v(u_i)-\alpha-\frac{1}{d})}p^{-\alpha}<p^{-\alpha}=r.\]
This shows that $zu\in B^n(r)$ for any $z\in \overline{B}(R)$. In particular, the function $f$ is analytic in $\overline B(R)$. This means that there is a power series $\sum_{n\geq 0}a_nx^n$ which is convergent in $\overline B(R)$ such that
$f(x)=\sum_{n\geq 0}{a_n}x^n$ for any $x\in \overline B(R)$. As usual for each $\kappa\in [0,R]$ we put
$$|f|_\kappa=\sup_{n\geq 0}|a_n|_p^n\kappa^n.$$
Define $R_i(x):=M_i(1,f_1(x),\ldots,f_N(x))$. We have
\begin{equation*}\begin{split}
f(z)&=(\Delta^{t}\Phi)(zu)\\
&=\sum_ip_i\big(\diff_{1}^{t_1}\cdots \diff_{k}^{t_k}R_i\big)(zu)\\
&=\sum_ip_i\big((x_{11}\partial_1+\cdots+x_{1n}\partial_{n})^{t_1}\cdots (x_{k1}\partial_1+\cdots+x_{kn}\partial_n)^{t_k}R_i\big)(zu)\\
\end{split}\end{equation*}
for $z\in\overline B(R)$.
By multinomial expansion again and by the ultrametric inequality one gets
$$|f(z)|_p\leq \max_{i}\max_{|t|<T}|(\Delta^{t}R_i)(zu)|_p.$$
On the other hand, Lemma \ref{ww} tells us that
$$(\Delta^{t}R_i)(zu)=Q_{i}(f_1(zu),\ldots,f_N(zu))$$
for polynomials $Q_i$ in $N$ variables such that $\cd^TQ_i$ have coefficients in $\mathcal O_K$. Since $zu\in B(r)$ from Section \ref{basis} we have
$|f_i(zu)|_p<1$ for $i=1,\ldots, N$. Hence
$$|(\Delta^{t}R_i)(zu)|_p\leq |\cd^{-1}|_p^{T}$$
which means that
$$|f(z)|_p\leq |\cd^{-1}|_p^{T},\quad\forall z\in \overline B(R).$$
In other words $|f|_R\leq |\cd^{-1}|_p^T.$
We now apply Lemma \ref{4} to the function $f$ to get
$$|f|_1\leq p^{-\frac{1}{2d}S_0T}|f|_R\leq |\cd^{-1}|_p^Tp^{-\frac{1}{2d}S_0T}.$$
Hence
$$|(\Delta^t \Phi)(su)|_p=|f(s)|_p\leq |f|_1\leq |\cd^{-1}|_p^Tp^{-\frac{1}{2d}S_0T}.$$
This leads to
\begin{equation}
\log|(\Delta^t \Phi)(su)|_p\ll -T\log|\cd|_p-S_0T\log p=-(S_0-e_L)T\log p
\end{equation}
for any $t\in\mathbb Z_{\geq 0}^d$ such that $|t|<T$.
We shall show that the order of $P$ along $V$ at any point $\gamma^s$ with $0\leq s<S$ is at least $T$. In fact, assume that it is not true at some point $\gamma^{s_0}$ with $0\leq s_0<S$, then by definition the exact order of $\Phi$ at $s_0u$ along $V_v$ is $T_0<T$, i.e. there is a $k$-tuple $t_0\in\mathbb Z_{\geq 0}^k$ such that $|t_0|=T_0$ and $(\Delta^{t_0}\Phi)(s_0u)\neq 0$. By Proposition \ref{lowerbound1} one has
$$\log|(\diff^{t_0}\Phi)(su)|_p\gg -(T_0(\cc+\log\cd+\log (D+T_0\ccc)+b)+DS^2h).$$
Combining this with $(3)$ we get
$$T(S_0-e_L)\log p\ll (T_0(\cc+\log\cd+\log (D+T_0\ccc)+b)+DS^2h).$$
This cannot hold if $c$ is large enough. Therefore $\Phi$ vanishes at any point of the set $\{su; 0\leq s<S_0\}$ of order at least $T$ along $V$. By the assumption $(G,V)$ is semistable and therefore \cite[Theorem 6.14]{aw} implies that there exists an integer $s'$ such that $0<s'<S$ and $\gamma^{s'}=e$, i.e. $\gamma$ is a torsion point, a contradiction. This ends the proof.

\subsection{Proof of Theorem \ref{past}}\label{propast}

The proof of Theorem \ref{past} is done by induction on $\dim G$. If $\dim G=1$ then $V$ must be equal to $\Lie(G)$, and we choose $H=G$. Now we assume that $\dim G>1$. By Proposition \ref{ss} we see that $(G,V)$ is not semistable. Lemma \ref{lss} shows that there is a proper quotient $\pi: G\rightarrow G'$ such that $(G',V')$ is semistable, where $V'=\pi_*(V)$. Let $\pi_F: G(F)\rightarrow G'(F)$ be the $F$-analytic homomorphism induced by $\pi$. It is known that the differential map $(\pi_F)_*=\pi\otimes_{\overline{\mathbb Q}}{\id}_F$.
By Section \ref{Log} we have that the following diagram
\begin{displaymath}
\xymatrix{
G(F)_f\ar[r]^{\pi_F}\ar[d]_{\log_{G(F)}}&
G'(F)_f\ar[d]^{\log_{G'(F)}}\\
\Lie(G(F))\ar[r]_{(\pi_F)_*}&\Lie(G'(F))}
\end{displaymath}
commutes, it follows that
$$\log_{G'(F)}(\pi_F(\gamma))=(\pi_F)*\big(\log_{G(F)}(\gamma)\big)\in (\pi_F)*(V\otimes_{\overline{\mathbb Q}}F)= \pi_*(V)\otimes_{\overline{\mathbb Q}}F=V'\otimes_{\overline{\mathbb Q}}F.$$
By Proposition \ref{ss} again we conclude that $\log_{G'(F)}(\pi_F(\gamma))$ must be $0$, this is equivalent to $\pi_F(\gamma)$ being a torsion point. Let $m$ be the order of $\pi_F(\gamma)$ and define $\widetilde G:=\ker\pi$. Then $\widetilde G$ is a commutative algebraic group defined over $\overline{\mathbb Q}$ and $\gamma^m\in \widetilde G(\overline{\mathbb Q})$ (since $\gamma\in G_f(\overline {\mathbb Q})$).
Therefore
$$\log_{\widetilde G(F)}(\gamma^m)
\in \Lie (\widetilde G(F))=\Lie(\widetilde G)\otimes_{\overline{\mathbb Q}}F$$
and
$$\log_{\widetilde G(F)}(\gamma^m)=\log_{G({F})}(\gamma^m)=m\log_{G(F)}(\gamma)\in V\otimes_{\overline{\mathbb Q}}F.$$
This gives
$$\log_{\widetilde G(F)}(\gamma^m)\subseteq (V\cap \Lie(\widetilde G))\otimes_{\overline{\mathbb Q}}F.$$
Furthermore we see that $\gamma^m\in \widetilde G(F)_f$. Since $\pi$ is proper it follows that the dimension of $\widetilde G$ is positive.
We apply the inductive hypothesis to the commutative algebraic group $\widetilde G$, the algebraic point $\gamma^m$ and the $\overline{\mathbb Q}$-linear subspace $V\cap \Lie(\widetilde G)$ of $\Lie(\widetilde G)$ to obtain an algebraic subgroup $\widetilde H$ of $\widetilde G$ of positive dimension
defined over $\overline{\mathbb Q}$ such that $\gamma^m\in \widetilde H(\overline{\mathbb Q})$ and such that $\Lie(\widetilde H)\subseteq \Lie(\widetilde G)\cap V\subseteq V$. Finally we take $H$ such that $H(\overline{\mathbb Q}):=\{\alpha\in G(\overline{\mathbb Q}); \alpha^m\in \widetilde{H}(\overline{\mathbb Q})\}$.
Then $H$ is an algebraic subgroup of $G$ of positive dimension defined over $\overline{\mathbb Q}$ and $\gamma\in H(\overline{\mathbb Q})$ and, in addition, $\Lie(H)\subseteq V$. This concludes the proof.

\subsection{Proof of Corollary \ref{cor1}}
The $F$-vector space $W\otimes_{\overline{\mathbb Q}}F$ is a subspace of $\Lie(G)\otimes_{\overline{\mathbb Q}}F$. In particular,
$\dim_{\overline{\mathbb Q}}W=\dim_F(W\otimes_{\overline{\mathbb Q}}F)$ is finite.
If $T$ contains only torsion points then the map $\log_{G(F)}$ vanishes on $T$. This means that $W$ is trival, and we choose $H=\{e\}$. Otherwise there exist a positive integer $m$ and non-torsion elements $\gamma_1,\ldots,\gamma_m$ in $T$ such that $\log_{G(F)}(\gamma_1),\ldots,\log_{G(F)}(\gamma_m)$ is a $\overline{\mathbb Q}$-basis for $W$. We apply successively Theorem \ref{past} with $\gamma=\gamma_i$, and obtain algebraic subgroups $H_i$ of $G$ defined over $\overline{\mathbb Q}$ such that $\gamma_i\in H_i$ and $\Lie(H_i)\subseteq W$ for $i=1,\ldots, m.$
Let $H=H_1\cdots H_m$ be the product of $H_1,\ldots,H_m$. This is an algebraic subgroup of $G$ defined over $\overline{\mathbb Q}$ since $G$ is commutative with
$$\Lie(H)=\Lie(H_1)+\cdots+\Lie(H_m).$$
Hence $\Lie(H)$ is a $\overline{\mathbb Q}$-vector subspace of $W$.
On the other hand
$$\log_{G(F)}(\gamma_i)=\log_{H_i(F)}(\gamma_i)\in \Lie(H_i(F))=\Lie(H_i)\otimes_{\overline{\mathbb Q}}F\subseteq \Lie(H)\otimes_{\overline{\mathbb Q}}F$$
for $i=1,\ldots,m$
and this implies that $W\otimes_{\overline{\mathbb Q}}F\subseteq \Lie(H)\otimes_{\overline{\mathbb Q}}F$. We conclude that
$$\dim_{\overline{\mathbb Q}}\Lie(H)=\dim_{\overline{\mathbb Q}}W,$$
i.e. $\Lie(H)=W$. This proves the corollary.

\subsection{Proof of Corollary \ref{cor2}}
By Section \ref{Log} we have that the following diagram
\begin{displaymath}
\xymatrix{
X(F)_f\ar[r]^{\varphi_F}\ar[d]_{\log_{X(F)}}&
G(F)_f\ar[d]^{\log_{G(F)}}\\
\Lie(X(F))\ar[r]_{(\varphi_F)_*}&\Lie(G(F))}
\end{displaymath}
is commutative. This gives
\begin{equation*}
\begin{split}
\log_{G(F)}(\varphi_F(x))&=(\varphi_F)_*(\log_{X(F)}(x))\subseteq (\varphi_F)_*(\Lie X(F))
\\
&=(\varphi_F)_*(\Lie(X)\otimes_{\overline{\mathbb Q}}F)=\varphi_{*}(\Lie(X))\otimes_{\overline{\mathbb Q}}F.
\end{split}
\end{equation*}
The corollary therefore follows from Theorem \ref{past}.

\section{Acknowledgments}

The authors are most grateful to Professor G. W\"ustholz for very insightful discussions and for constant encouragement to work on the topic. Moreover, they are grateful to an anonymous referee for many helpful remarks. The first author was supported by Austrian Science Fund (FWF): P24574. The second author was supported by grant PDFMP2\_122850 funded by the Swiss National Science Foundation (SNSF).

\noindent{\sc Clemens Fuchs}\\
Department of Mathematics\\ University of Salzburg\\ Hellbrunnerstr. 34\\ 5020 Salzburg\\ Austria\\
Email: {\sf clemens.fuchs@sbg.ac.at}\\

\noindent{\sc Duc Hiep Pham}\\
Department of Mathematics\\
Hanoi National University of Education\\ 136 Xuan Thuy, Cau Giay, Hanoi\\ Vietnam\\
Email: {\sf phamduchiepk6@gmail.com}


\begin{thebibliography}{9}
\bibitem{aw} A. Baker and G. W\"ustholz, \emph{Logarithmic forms and Diophantine geometry}, New Mathematics Monographs, vol. {\bf 9}, Cambridge University Press, 2007.
\bibitem{bv} E. Bombieri and J. Vaaler, \emph{On Siegel's lemma}, Invent. Math. {\bf 73} (1983), 11-32.
\bibitem{Bourbaki} N. Bourbaki, \emph{Elements of Mathematics. Lie groups and Lie algebras. Part I:  Chapters {\bf 1-3}. English translantion.,} Actualities scientifiques et industrielles, Herman. Adiwes International Series in Mathematics. Paris: Hermann, Publishers in Arts and Science; Reading, Mass.: Addison-Wesley Publishing Company. XVII, 1975.
\bibitem{bertrand} D. Bertrand, \emph{Lemmes de z\`eros et nombres transcendants}, S\'eminaire Bourbaki Vol. 1985/86, Ast\'erisque {\bf 145-146} (1987), 21-44.
\bibitem{fw} G. Faltings and G. W\"ustholz, \emph{Einbettungen kommutativer algebraischer Gruppen und einige Eigenschaften}, J. Reine Angew. Math. {\bf 354} (1984), 175-205.
\bibitem{fuchs-pham} C. Fuchs and D. H. Pham, \emph{Commutative algebraic groups and $p$-adic linear forms}, http://arxiv.org/pdf/1404.4209v1.pdf.
\bibitem{matev} T. Matev, \emph{The p-adic analytic subgroup theorem and applications}, http://arxiv.org/pdf/1010.3156v1.pdf.
\bibitem{se} J. P. Serre, \emph {Quelques propri\'et\'es des groupes alg\'ebriques commutatifs}, Ast\'erisque {\bf 69-70}, 1979.
\bibitem{w1} G. W\"ustholz, \emph{ Algebraische Punkte auf Analytischen Untergruppen algebraischer Gruppen}, Ann. of Math. {\bf 129} (1989), 501-517.
\bibitem{z} Y. G. Zarhin, \emph{$p$-adic abelian integrals and commutative Lie groups}, J. Math. Sci. {\bf 81} (1996), no. 3, 2744-2750.
\vspace*{0.2cm}
\end{thebibliography}
\end{document}